\documentclass[11pt]{article}
\usepackage{amsmath, amsfonts, amssymb,dsfont,mathrsfs}
\usepackage{graphicx, amsthm,color,url}

\newtheorem{thm}{Theorem}[section]
\newtheorem{lemma}[thm]{Lemma}
\newtheorem{proposition}[thm]{Proposition}
\newtheorem{definition}[thm]{Definition}

\theoremstyle{definition}
\newtheorem{example}[thm]{Example}
\theoremstyle{remark}
\newtheorem{remark}[thm]{Remark}
\numberwithin{equation}{section}

\newcommand{\R}{\mathbb{R}}

\newcommand{\N}{\mathbb{N}}

\newcommand{\E}{\mathbb{E}}
\renewcommand{\P}{\mathbb{P}}

\providecommand{\tto}{\mathop{\rightrightarrows}\nolimits}

\usepackage{tikz}
\usepackage{tikz-3dplot}
\usetikzlibrary{calc,patterns,angles,quotes}
\usetikzlibrary{babel}
\usepackage{pgfplots}
\usepgfplotslibrary{patchplots}
\usetikzlibrary{backgrounds, intersections}
\usepgfplotslibrary{fillbetween}
\usetikzlibrary{arrows,automata}
\usepackage{subcaption}

\pgfplotsset{compat=1.17}
\usepgfplotslibrary{colorbrewer}

\usepackage{enumitem}

\definecolor{forestgreenweb}{rgb}{0.13, 0.55, 0.13}

\newcommand{\DSr}[1]{{#1}}
\newcommand{\GL}[1]{{{#1}}}

\begin{document}
	
	\begin{center}
		{\LARGE The value of Shared Information for allocation of drivers in ride-hailing: a proof-of-concept study.}
		
		\vspace{1cm}
		
		{\Large \textsc{Gianfranco Liberona, David Salas, Léonard von Niederh\"{a}usern}}
	\end{center}
	
	\bigskip
	
	\noindent\textbf{Abstract.} For drivers in ride-hailing companies, allocation within the city is paramount to get matched with rides. This decision depends on many factors, where some of them (such as demand and allocation of others) are unknown for the drivers, but are available for the company. In this work, we investigate whether it is beneficial or not for the ride-hailing company to share this information with their drivers. To do so, we study the problem through the lens of Stackelberg games, and we propose a new indicator called the \emph{Expected Value of Shared Information}. We present a simplified model to conduct a proof-of-concept study: we provide explicit single-level reformulations of the bilevel programming problems derived from the model, and perform several simulations with randomly generated data. Our preliminary results suggest that sharing information could be beneficial and deserves to be further studied. \bigskip
	
	\noindent\textbf{Key words.} Stochastic Programming; ride-hailing; Stackelberg Game; Expected Value of Shared Information.

	\section{Introduction \label{sec:intro}}

	The growing popularity of ride-hailing companies, such as Uber and Lyft, has changed the way we move around the city. There is a new relation between passengers and drivers, which now interact \GL{throughout} this new third party.
	Several new problems have arisen from this context, such as spatio-temporal pricing \cite{bimpikis2019spatial}, reallocation of resources \cite{balseiro2020dynamic,he2020robust}, or online matching \cite{housni2021matching} (see, e.g., \cite{wang2019ridesourcing,benjaafar2020operations} for some recent surveys). Here, we are interested in the way information affects the relation between a ride-hailing company and its drivers. \medskip
	
	To understand this relation, let us describe the general framework we are set in, which is motivated by recent literature \cite{wang2019ridesourcing,benjaafar2020operations,bimpikis2019spatial}. First, a city can be understood as a network of interconnected locations, to which drivers are allocated. At every given time, new passengers appear in the locations, requesting a ride. The ride-hailing company then matches each passenger with a driver in the same location, and receives a compensation proportional to the cost of the ride. While the compensation can be assumed to be constant, the company has the liberty to adapt prices, generating different fares depending on the location and time. \medskip
	
	Of course, surge pricing affects the demand. But more interesting for us, spatial pricing (different fares between locations) can induce reallocation of unmatched drivers. Indeed, a particularity of the ride-hailing companies is that they do not employ drivers, but rather they consider drivers as independent operators using the matching service. Thus, drivers are free to reallocate themselves whenever they consider it convenient. \medskip
	
	Some key elements for an unmatched driver to decide whether to change location or not are the following: the available ride fares, the costs of reallocation, the number of demanded rides at each location, and the number of previously matched drivers arriving to each location (who can be matched to other passengers as soon as they finish their previous rides). The first two are known information for the drivers, but the demand (exogenous uncertainty) and the previously matched drivers (endogenous uncertainty) are not. On the other hand, at each stage of reallocation and matching, the ride-hailing company can forecast the exogenous uncertainty, and has all the information available for the endogenous one.\medskip
	
	\GL{The goal of this work is to assess whether the ride-hailing company can benefit from sharing information} with the unmatched drivers. This behavior is observed nowadays, where ride-hailing companies provide some demand information to the drivers beyond spatial pricing (see, e.g., the Uber's driver-app description in \cite{UberApp2021}).\medskip
	
	The notion of sharing information has been studied before, for example in the context of supply chains \cite{TEUNTER20181044,LIU2021572}, network restoration \cite{SHARKEY2015309} and pricing problems \cite{SUN2021511}, where different parties on each context, can be benefited from sharing information between them. Here, we propose to study the problem of value of information through the lens of Stackelberg games (see, e.g., \cite{dempe2002foundations,DempeEtAl2015}).\medskip
	
	On the one hand, the company acts as the leader, deciding the spatial prices. On the other hand, the drivers act as followers, solving a stochastic allocation equilibrium problem. \DSr{The main variant in this work is that we consider the leader to have perfect information, due to advanced forecasting of exogenous uncertainty, and asymmetry of information with respect to the follower for endogenous uncertainty. This variant distinguishes our problem from classic stochastic Stackelberg games, where the leader usually decides in a  \textit{here-and-now} fashion, prior \GL{to the revelation of} uncertainty, and the follower decides after the revelation of uncertainty, in a \textit{wait-and-see} fashion (see, e.g., \cite{BurtscheidtClaus2020}). }\medskip
	
	Inspired by the classic Expected Value of Perfect Information in (single level) stochastic programming (see, e.g., \cite[Chapter 4]{Birge_Introduction_2011}), we propose to study the \textit{Expected Value of Shared Information (EVSI)}, which is a new indicator that measures the impact, for the leader, of forecasting and then sharing the perfect information with the drivers. In a nutshell, the EVSI compares two problems: when only the leader has perfect information, and when both agents have perfect information. Since the information is modeled as the realization of a random variable, this comparison is done in average, assuming that the game is repeated many times.  \medskip

	\DSr{Our main motivation is to understand the economic value of information sharing in the context of ride-hailing companies, and so we start in Section \ref{sec:Model} by presenting the detailed model we work with, which is the repeated one-stage pricing problem of the ride-hailing company, coupled with the reallocation problem for the unmatched drivers. At this stage, we consider a simplified version of the problem, since we are interested in the proof-of-concept of the EVSI indicator. Since this indicator can be applied in the general framework of stochastic bilevel games,  In Section \ref{sec:Parametric} we develop the information sharing concept in an abstract mathematical setting.  The next two sections are devoted to \GL{computing} the EVSI for our model: In Section \ref{sec:Reformulation}, we describe how we can reduce the various bilevel programming problems we encounter to single mixed-integer bilinear programming problems, which allows us to efficiently find the solutions for different scenarios. In Section \ref{sec:Numerical}, we provide a numerical analysis with randomly generated data, which suggests that sharing information is indeed beneficial in this problem.} We finish with some conclusions and perspectives in Section \ref{sec:Final}.

	
	\section{The allocation problem in ride-hailing: a simplified model\label{sec:Model}}
	
	\DSr{In this section, we present the model we will study \GL{throughout this} work. While literature is vast in how to model the drivers dynamics in ride-hailing, several studies consider reduced models due to the complexity of the real problem (see, e.g., \cite{bimpikis2019spatial,Castillo2017GooseChase}). In this work, our main attention is the asymmetry of information between the company and the drivers, and so, we do several simplifying assumptions in other dimensions, in order to obtain a tractable problem.}  
	
	\subsection{Model and simplifying assumptions}
	
	Let us consider the following situation: at a certain moment, a driver associated with a ride-hailing company that has not been matched with a passenger must decide whether to keep searching for a match around \DSr{its} current location, or to move to another one within the city. We model the different locations as a finite set of zones, $I = \{1,\ldots,n\}$, connected as a directed graph. If the driver is in the $i$th zone, \DSr{its} reallocation decision will depend on five factors:
	
	\begin{enumerate}
		\item The vector of (spatial) prices fixed by ride-hailing company, $p = (p_i\ :\ i\in I)$.
		\item The vector of demands of each zone $d = (d_i\ :\ i\in I)$.
		\item The costs of moving to another zone, $\alpha_i = (\alpha_{ij}\ :\ j\in I)$. Of course, $\alpha_{ii} = 0$.
		\item The vector of previously unmatched drivers that will be at each node, $x = (x_i\ :\ i\in I)$.
		\item The vector of occupied drivers who will arrive at each node (and will become available at that node), $y=(y_i\ :\ i\in I)$. 
	\end{enumerate}

		\begin{remark} The vector $p$ \GL{is} considered to be the (surge) prices of an average trip, in order to avoid time-length of trips into the analysis. This is a common practice in literature (see, e.g., \cite{Castillo2017GooseChase,Yan2020Dynamic}). Similarly, the reallocation \GL{costs} are consider to be average, consistently with the discretized model of zones. 
			
		\end{remark}
	
	\paragraph{Single-stage model}  
	
	Several studies (see, e.g., \cite{bimpikis2019spatial,balseiro2020dynamic})  consider a dynamic multi-stage problem where: 1) the demand distribution varies from one stage to the next, and they are correlated temporally; 2) at each stage, drivers can get in and out from the app, changing the total labor force; 3) Drivers take into account future estimations of the demand to reallocate.\medskip
		
		In this work, we consider a much more simple situation, where the company and all drivers play a single-stage that repeats many times, and where each repetition is independent \GL{from} the others. Moreover, we assume that the the number of unmatched drivers is fixed, and that drivers only consider the current information to decide. This assumption led us to model occupied drivers as an independent random variable, instead of \GL{a} dynamic stochastic process, as we will see later on. \medskip
		
		While such \GL{a} situation is not realistic, this static model has theoretical value. Independent repeated stages are used as the baseline before developing more complex dynamic models. As we will see in the following sections, in terms of sharing information, the static model had already several difficulties that we addressed during this work.\medskip
		
		In what follows, we will consider random events as random variables defined over a single measurable space $(\Omega,\mathscr{F})$, representing the state of nature. We denote by $\omega$ an arbitrary element of $\Omega$. 
	
	\paragraph{Demand and willingness-to-pay}
	
	\DSr{We consider that the demand on each zone $i\in I$, depends on two factors: a random variable which models people requesting a ride, and the marginal price $p_i$, which induces a reduction associated with the willingness-to-pay of riders (see, e.g., \cite{Castillo2017GooseChase}). Here we omit the effects of waiting times, considering zones as a single node.}\medskip
	
	In this work, we model this dependency as a nominal value $d_{0,i}(\omega)$ which represents the demand for the minimal price $p_{i,\min}$, multiplied by a linear discount factor depending on the price:
	\begin{equation}\label{eq:Def-Demand}
		d_i = d_i(p_i,\omega) = d_{0,i}(\omega)\left(1-\delta\frac{p_i - p_{i,\min}}{p_{i,\max} - p_{i,\min}}\right).
	\end{equation}
The willingness-to-pay and the willingness-to-wait are two factors that reduce the effective demand in ride-hailing. In the literature, these discount factors are considered to be exponential-like functions (see, e.g., \cite{Yan2020Dynamic}), and therefore highly nonlinear. Thus, \GL{the decision is} to conduct \GL{this} study with a simpler linear discount model depending only on prices. \GL{Piecewise linear models} can be used to approximate the exponential model, and so this study can be used as a baseline for more accurate discount functions.
	
	\paragraph{Common belief on uncertain parameters}
	
	Since each driver has limited observability about the other drivers, the value of $y=(y_i\ :\ i\in I)$ is uncertain, even though it is known information for the ride-hailing company. Thus, from the unmatched drivers' perspective, $y=y(\omega)$ is also a random variable.\medskip
	
	\DSr{We will assume that every driver considers that the random variable $\xi(\omega) = (y(\omega),d_0(\omega))$ follows a common distribution $D_F$. Then, to decide \GL{whether} to reallocate or not,} each driver must solve the following optimization problem:
	\begin{equation}\label{eq:Problem-OneDriver}
		\max_{j\in I} \,\mathbb{E}_{\xi\sim D_F}\left[ p_j\min\left(\frac{d_j}{x_j+y_j},1\right) \right] - \alpha_{ij},
	\end{equation}
	where $x_j+y_j$ is the amount of available drivers in zone $j$, and the value $\min\left(\frac{d_j}{x_j+y_j},1\right)$ represents the probability of being matched in zone $j$. \DSr{On the one hand, if $d_j\geq x_j+y_j$, then the driver will be matched. On the other hand, if $d_j< x_j+y_j$, the probability of being matched coincides with $d_j/(x_j+y_j)$, assuming that in such a case, all passengers will be matched.} 
	
	\paragraph{Single decision-maker}
	
	\DSr{\GL{In this study}, we model all unmatched drivers as a single new follower, who aims to maximize the social welfare of all drivers.} We will assume that only unmatched drivers report to this central decision-maker, while matched drivers become unavailable. Thus, the follower must decide the allocation of unmatched drivers $x=(x_i\ :\ i\in I)$ while the vector $y = (y_i\ :\ i\in I)$ is uncertain. \DSr{We maintain the assumption that the single decision-maker assumes that $\xi = (y,d_0)\sim D_F$.}\medskip
	
	\DSr{In practice, we should consider that drivers interact in a game with imperfect information, that should lead to a single-leader multi-follower game (see \cite{HuFukushima2015}), with the caveat that the lower-level game should be a Bayesian game (see, e.g., \cite{Tadelis2013Game}). However, such a model poses very challenging theoretical difficulties, and the theory is considerably underdeveloped in comparison to (one-leader-one-follower) Stackelberg games (see, e.g., \cite{AusselSvensson2020Short}).}\medskip
	
	\DSr{Other (more classic) simplifications have been considered in the literature, such as seeing drivers as part of a continuum and model their interaction as a single flow equilibrium problem (see, e.g., \cite{bimpikis2019spatial}). We prefer to consider one aggregating follower since in this model, we still can study the value of sharing information.  Moreover, up to a certain scale, a single follower captures the situation where drivers can communicate between them outside the ride-hailing platform, and they can coordinate their allocation (see, e.g.,  \cite{tripathy2022driver}).}
	
	\subsection{Stackelberg problem formulation}
	
	To model the decision process of the single follower, let us assume that there is an amount of $N_0$ drivers unmatched, with initial allocation
	$x_0 = (x_{0,1},
	\ldots, x_{0,n})\in \R^n_+$.
	Let us define the variable $v_{ij}$ as the amount of unmatched drivers who will change from zone $i$ to zone $j$, and let $v$ be the matrix that collects all this information. In this context, $v_{ii} = 0$ for all $i\in I$. Then, we can compute a reallocation $x$ in terms of the displacement matrix $v$ simply as:
	\begin{equation}\label{eq:Kirchoff}
		x_{j}(v) = x_{0,j} + \sum_{i\neq j} v_{ij} - \sum_{k\neq j} v_{jk},\quad\forall j\in I.
	\end{equation}
	\GL{Throughout this} work, we will assume that
	$x_{0} = (x_{0,1},\ldots,x_{0,n})>0$. If for some zone $i\in I$ one has that $x_{0,i} = 0$, the developments can be easily adapted by removing the variables $\{v_{i,j}\ : j\neq i\}$. Then, for a given price vector $p$, the aggregated allocation problem is posed as follows:
	\[
	F(p) := \begin{cases}
		\displaystyle \max_{\GL{v}} & \displaystyle \sum_{i=1}^n c p_i\mathbb{E}_{\DSr{\xi\sim D_F}}\left[\min(x_i(v)+y_i,d_i)\right] - \sum_{i\neq j} \alpha_{ij}v_{ij}\\
		\mbox{s.t.}\quad&\begin{cases}
			v\geq 0\\
			\displaystyle \sum_{k\neq j}v_{jk} \leq x_{0,j},\quad \forall j\in I\\
		\end{cases}
	\end{cases},
	\]
	where $c\in(0,1)$ is the fraction of the ride price that the driver gets. \medskip
	
	\GL{Throughout} this work, we consider that $D_F$ is given as the following product distribution: 1) the vector of previously matched drivers who will arrive at each node, follow a uniform distribution, that is, $y_j\sim U(0,\bar{y})$ where $\bar{y}$ is a constant value for all nodes. This distribution represents the lack of information for the unmatched drivers about the matched ones; and 2) the vector of nominal demand as a discrete one, considering $m\in\N$ feasible scenarios. \DSr{That is, we consider $\omega_1,\ldots,\omega_m$ scenarios, with probabilities $\rho_1,\ldots,\rho_m$, respectively. For each $k\in\{1,\ldots,m\}$, the demand $d$ is given by $d(p,\omega_k) = (d_i(p_i,\omega_k))_{i\in I} = (d_{i,k})_{i\in I}$. Observe that, in order to simplify the exposition, we have omitted the dependency of $p$ in the expression $(d_{i,k})_{i\in I}$.} Thus, we can write
	\begin{equation}\label{eq:ScenariosDecomposition}
		\begin{aligned}
			\mathbb{E}_{\xi\sim D_F}\left[p_i\min(x_i(v)+y_i,d_i)\right] &= -\E_{\xi\sim D_F}(p_i\max(-x_i(v)-y_i,-d_i))\\
			&= -\sum_{k=1}^m p_i\E_{y\sim U}[\max(-x_i(v)-y_i,-d_{i,k})]\cdot \rho_k\GL{.}
		\end{aligned}
	\end{equation}
	Therefore, the follower will deal with a discrete version of its original problem $F(p)$, given by 
	\begin{equation}\label{eq:FollowerProblem-DiscreteDistribution}
		\small
		F_m(p) := \begin{cases}
			\displaystyle \min_{\GL{v}} & c \displaystyle \sum_{i=1}^n \sum_{k=1}^m p_i\E_{y\sim U}[\max(-x_i(v)-y_i,-d_{i,k})]\cdot \rho_k + \sum_{i\neq j} \alpha_{ij}v_{ij}\\
			\\
			\mbox{s.t.}\quad&\begin{cases}
				-v\leq 0\\
				\displaystyle \sum_{j\neq i}v_{ij} - x_{0,i}\leq 0,\quad \forall i\in I
			\end{cases}
		\end{cases} \GL{.}
	\end{equation}
	
	Now, the ride-hailing company must decide the price vector $p$. The company does not necessarily know the exact value of the demand vector $d$, but has a good forecasting capacity of it (see, e.g, \cite{UberForecasting2018,LyftForecasting2019}). Also, it knows the vector $y$ of occupied drivers. Since the company aims to maximize its revenues, for each realization of $\xi = (d_0,y)$, it must solve the following bilevel programming problem:
	\begin{equation}\label{eq:Leaders-Problem-WS}
		L(d_0,y) := \begin{cases}
			\displaystyle \max_{\GL{p,v}} & \displaystyle \sum_{i=1}^n (1-c) p_i\cdot \min(x_i(v)+y_i,d_i)\\
			\mbox{s.t.}\quad&\begin{cases}
				p_i \in [ p_{i,\min}, p_{i,\max} ],\quad \forall i\in I\\
				v \mbox{ solves } F_m(p).
			\end{cases}
		\end{cases}   
	\end{equation}

	\DSr{The leader and the follower \GL{solve their respective} problems several times, since this is a game that repeats constantly. Thus, the asymmetry of information in this model comes from the fact that the leader observes the realization of $\xi = (d_0,y)$ each time to decide the surge pricing. Sharing information (as we will see in the next section), would mean to reveal the observed value of $\xi$ with the drivers. Naturally, this should be a long-term policy and so, the leader is also interested in the average revenues obtained by solving  \eqref{eq:Leaders-Problem-WS} repeated times. \medskip
		
		Here, we consider that $\xi$ might follow a different distribution, $\xi\sim D_L$, which is known by the leader. This consideration comes from the fact that the ride-hailing companies collect historical data on demand and ratio of occupancy of the drivers, allowing them to obtain an accurate model for the distribution $D_L$. In contrast, drivers only have indirect information on \GL{these} parameters, getting a much rougher model on the behavior of $\xi$.\medskip
		
		In this context, for example, distribution of $d_0(\omega)$ should be a multivariate normal-like distribution around a nominal value $\bar{d}_0 = (\bar{d}_{0,i}\ :\ i\in I)$, for which the scenarios $\omega_1,\ldots,\omega_m$ provide a piecewise constant approximation of it.}
	
	\begin{remark} In what follows, to ease the notation, we will omit the dependency $x_i(v)$ given by \eqref{eq:Kirchoff}, and simply write $x_i$ instead. The reader should keep in mind, though, that $(x_i\ :\ i\in I)$ is not a dependent variable, but rather an expression in terms of the variables $(v_{i,j}\ :\ i\neq j\in I)$. In particular, the leader's decision $p$ and the follower's problem are coupled by the product $p_ix_i$, which gives the formulation its bilevel structure.
	\end{remark}
	
	
	\section{Parametric Bilevel Problems and Expected Value of Shared Information\label{sec:Parametric}}
	
	\DSr{In stochastic bilevel optimization, the structure of Problem \eqref{eq:Leaders-Problem-WS} has not gotten too much attention. In general, the structure of stochastic bilevel optimization is, as understood in the literature (see, e.g., \cite{BurtscheidtClaus2020}), that the leader decides \textit{here-and-now}, and the follower decides \textit{wait-and-see}, that is:
		\[
		\text{Leader decides }p\to \xi=(y,d_0)\text{ is revealed}\to\text{Follower decides } v.
		\]
		This is the natural structure for a Stackelberg game, since the \GL{follower} reacts to the \GL{leader}'s decision, and so it is usually understood that the \GL{leader} decides first. However, this sequence is the opposite \GL{of} the information structure that we have in Problem \eqref{eq:Leaders-Problem-WS}: in our problem \GL{it is} the \GL{leader} that decides \textit{wait-and-see} and the \GL{follower} decides \textit{here-and-now}. To make both aspects compatible, we consider that, prior to the revelation of the random variable, the \GL{follower} \textit{commits} to a policy $p\in[p_{\min},p_{\max}]\mapsto v(p)$ on how to react to the prices $p$, prior the revelation of the random variable $\xi$. This leads to the following decision sequence:
		\begin{align}
			\small \text{Follower commits to }v(\cdot)&\small\to \xi=(y,d_0)\text{ is revealed}\notag\\
			&\small\to\text{Leader decides } p\to \text{Follower implements }v(p).\label{eq:schemeinfo}
		\end{align}
		
		In this section, we present the abstract model that we will use to describe the interaction between a ride-hailing company and its drivers, following this latter decision sequence. The setting fits into the general framework of stochastic bilevel optimization \GL{but, as described before}, with the caveat that randomness is revealed after the followers' commitment, differing from the usual concept of the problem (see, e.g., \cite{BurtscheidtClaus2020}).}
	
	
	\subsection{Here-and-now \GL{follower}'s formulation}
	We consider the following optimistic parametric bilevel programming problem
	
	\begin{equation}\label{eq:Def-ParametricProblem}
		\varphi(\xi) = \left\{ 
		\begin{array}{l}
			\displaystyle\min_{p,v}\,\,  \theta(p,v,\xi)\\
			\text{s.t. }\left\{ \begin{array}{l l}
				p\in X &  \\                
				v\text{ solves } \left\{\begin{array}{l}
					\displaystyle\min_{\GL{v}} f\left(p,v,\xi\right) \\
					\text{ s.t. } v\in V(p),
				\end{array}\right.
			\end{array} \right.
		\end{array}\right.
	\end{equation}
	where $\xi \in \Xi\subset\R^{k}$  is the parameters' vector. Here, the function $\varphi:\Xi\to\R$  is the value function of Problem (\ref{eq:Def-ParametricProblem}). For each vector $\xi\in \Xi$, the leader aims to minimize the loss function $\theta:\R^n\times\R^m\times\R^{k}\to \R$. It only controls the first variable $p\in X\subset\R^n$, which we call the leader's decision. The set of admissible leader's decisions $X\subset\R^n$ is fixed.\medskip
	
	Similarly, for each $\xi\in\Xi$ and each leader's decision $p\in X$, the follower aims to minimize the loss function $f:\R^n\times\R^m\times \R^{k}\to \R$. It only controls the second variable $v\in V(p)\subset\R^m$, which we call the follower's decision. The set of admissible decisions $V(p)\subset \R^m$ depends on the leader's decision $p$, inducing a set-valued map $V: X\tto \R^m$. The range of $V$ is contained in an ambient set $\bar{V}\subset\R^m$, that is,
	\[
	R(V):= \bigcup_{p\in X} V(p)\subset \bar{V}.
	\]
	In what follows, we consider the following (standard) assumptions over Problem (\ref{eq:Def-ParametricProblem}):
	\begin{enumerate}[label=(\emph{H\arabic*}), ref=(H\arabic*)]
		\item\label{hyp:AmbientConvexCompact} The sets $X$ and $\bar{V}$ are nonempty, convex and compact, and $\Xi$ is nonempty and closed.
		\item\label{hyp:Continuity} The loss functions $\theta$ and $f$ are continuous.
		\item\label{hyp:ConstraintMapContinuous} The set-valued map $V$ has nonempty convex compact values, and it is both upper and lower semicontinuous (in the sense of Painlevé-Kuratoski and Berge. See, e.g., \cite{Ichiishi1983,AubinFrankowska1990}).
	\end{enumerate}
	
	Under this framework, which is fairly general, one can ensure the existence of solutions of the parametric Problem \eqref{eq:Def-ParametricProblem}. This existence result is classic in the literature (see, e.g., \cite{dempe2002foundations}), but we recall it for completeness.
	
	\begin{lemma}\label{lem:Existence} Assume hypotheses \ref{hyp:AmbientConvexCompact}-\ref{hyp:ConstraintMapContinuous}. Then, for each $\xi\in\Xi$, Problem \eqref{eq:Def-ParametricProblem} admits at least one solution. Furthermore, the value function $\varphi:\Xi\to\R$ is lower semicontinuous.
	\end{lemma}
	\begin{proof}
		Let us denote $S:X\times \Xi\tto \bar{V}$ given by
		\[
		S(p,\xi) = \arg \min_{v\in V(p)} f(p,v,\xi)
		\]
		Using the well-known Berge Maximum Theorem (see, e.g., \cite[Theorem 2.3.1]{Ichiishi1983}), hypotheses \ref{hyp:AmbientConvexCompact}, \ref{hyp:Continuity} and \ref{hyp:ConstraintMapContinuous} entail that $S$ is upper-semicontinuous and closed-valued. Thus, $\mathrm{gph}(S)$ is closed (see, e.g., \cite[Theorem 2.2.1]{Ichiishi1983}). Let us define $K~:~\Xi\tto~X\times\bar{V}$ given by
		\[
		K(\xi) = \{(p,v)\in X\times\bar{V}\ :\ (p,v,\xi)\in \mathrm{gph}(S)\}.
		\]
		Noting that $K(\xi)$ is compact for every $\xi$ (since it is closed and a subset of $X\times \bar{V}$), and that Problem \eqref{eq:Def-ParametricProblem} can be written as 
		\[
		\varphi(\xi) = \min_{p,v} \{\theta(p,v,\xi)\ :\ (p,v)\in K(\xi)\},
		\]
		we deduce from Weierstrass theorem that Problem \eqref{eq:Def-ParametricProblem} has a solution for every fixed $\xi\in \Xi$. 
		
		Note that $\mathrm{gph}(K)$ coincides with $\mathrm{gph}(S)$ after permuting $(\xi,p,v)$ to $(p,\xi,v)$. Since $\mathrm{gph} S$ is closed and $K$ is compact-valued,  we deduce that $K$ is upper-semicontinuous (see, e.g., \cite[Theorem 2.2.3]{Ichiishi1983}). Moreover, since $\theta$ is continuous, we can apply \cite[Theorem 1.4.16]{AubinFrankowska1990} to conclude that the marginal function
		$$\xi\mapsto \sup_{(p,v)\in K(\xi)} -\theta(p,v,\xi)$$
		is upper semicontinuous. The result follows by noting that
		\[
		\varphi(\xi) = \inf_{(p,v)\in K(\xi)} \theta(p,v,\xi) = - \sup_{(p,v)\in K(\xi)} -\theta(p,v,\xi),
		\]
		which is therefore lower semicontinuous.
	\end{proof}

	\DSr{Uncertainty in our setting consist in considering the parameter $\xi = \xi(\omega)$ as a random variable over a measurable space $(\Omega,\mathscr{F})$. We allow the leader and the follower to have different beliefs about how $\xi$ is distributed. Thus, the distribution $D_L$ used by the leader to model $\xi$ might (and usually should) differ from the distribution of $D_F$ used by the follower. In other words, the leader works with a probability measure $\P_L$ over $(\Omega,\mathscr{F})$, which differs from the probability measure $\P_F$ used by the follower.}\medskip
	
	We consider that the follower solves a \emph{here-and-now} problem, taking into account the leader's decision $p$ as a parameter and the value of $\xi(\omega)$ as uncertain. Thus, for each leader's decision $p$, the follower is solving the  problem
	\begin{align}\label{eq:follower-nonparam} 
		\left\{\begin{array}{l}
			\displaystyle\min_v \quad  \E_{\xi\sim D_F}\left[f\left(p,v,\xi\right)\right] \\
			\text{ s.t. } \quad v \in V(p).
		\end{array}\right.
	\end{align}
	We refer by $D(p)$ to the solution set of Problem \eqref{eq:follower-nonparam}.
	
	
	\subsection{Measuring the value of perfect information: Expected Value of Shared Information}
	
	Following the information scheme \eqref{eq:schemeinfo}, for each realization of $\xi$, the \GL{leader} is solving the following parametric problem: 
	\begin{equation}\label{eq:psi}
		\psi(\xi) := \left\{ \begin{array}{l}
			\displaystyle\min_{p,v}\,\,  \theta\left(p,v,\xi\right)\\
			\text{s.t. }\left\{ \begin{array}{l l}
				p\in X &  \\
				v\in D(p),&
			\end{array} \right.
		\end{array}\right.
	\end{equation}
	From the leader's perspective, for each decision vector $p$, the follower's optimal response is a deterministic point $v\in D(p)$. Using the notation of the classic stochastic optimization (see, e.g., \cite{Birge_Introduction_2011}) , we consider the next definition to measure the value of the optimal decision process under perfect information.
	\begin{definition}[WS]\label{def:EVPI} We define the \emph{Wait-and-See Value (WS)} as the expected value of the value function $\psi$ of Problem \eqref{eq:psi}. That is, 
		\begin{equation}
			WS:= \E_{\xi\sim D_L}(\psi) = \int_{\Omega} \psi(\xi(\omega))d\P_L(\omega).
		\end{equation}
	\end{definition}
	
	Considering the interpretation of Problem (\ref{eq:Def-ParametricProblem}) as a parametric problem with recourse, the leader has another option in Stackelberg games under perfect information: it might \emph{share} this information with the follower. This alternative comes from the fact that the follower is an independent agent, which reacts to new information. In order to measure the value of sharing perfect information, we introduce the following definition:
	
	\begin{definition}[SWS and EVSI]\label{def:EVSI} We define the \emph{Shared Wait-and-See Value (SWS)} as the expected value of the value function $\varphi$ from the parametric problem \eqref{eq:Def-ParametricProblem}.
		That is, 
		\begin{equation}\label{eq:Def-SWS}
			SWS:= \E_{\xi\sim D_L}(\varphi) = \int_{\Omega} \varphi(\xi_1(\omega),\xi_2(\omega))d\P(\omega),
		\end{equation}
		The \emph{Expected Value of Shared Information (EVSI)} is then defined as 
		\[EVSI:= WS-SWS,\]
		which measures the gain or loss of the leader having perfect information and sharing it with the follower. 
	\end{definition}

	While perfect information is always beneficial for the leader, there is no such universal relation between sharing it with the follower or not. Thus, when evaluating the value of perfect information for the leader, both WS and SWS should be  computed. This is illustrated in the following example.
	
	
	\begin{example} Let us consider only an exogenous random event $\xi$ to be a fair Bernoulli trial and the indicator functions $\delta_0$ and $\delta_1$ given by
		\[
		\delta_i(\xi) = \begin{cases}
			1&\quad \mbox{ if }\xi = i\\
			0&\quad\mbox{ otherwise,}
		\end{cases}\qquad\mbox{ for }i=0,1.
		\]
		Let the leader's decision set to be $X=[0,1]$ and the follower's decision set to be $V(p) = [0,1]$, for all $x\in X$. Let the follower's loss function to be
		\[
		f(v,\xi(\omega)) := v^2\delta_0(\xi(\omega)) + (1-v)^2\delta_1(\xi(\omega)).
		\]
		Assume that $D_L = D_F$ and consider two possible loss functions for the leader:
		\begin{align*}
			\theta_+(p,v,\xi) &= \frac{1}{2}\big(p^2\delta_0(\xi) + (1-p)^2\delta_1(\xi)\big) + \GL{f(v,\xi)},\\
			\theta_{-}(p,v,\xi) &= \frac{1}{2}\big(p^2\delta_0(\xi) - (1-p)^2\delta_1(\xi)\big) - \GL{f(v,\xi)}.
		\end{align*}
		We consider then two problems:
		\[
		\begin{array}{|c|c|}
			\hline
			&\\
			\begin{array}{l}
				\displaystyle\min_{p,v}\,\,  \theta_+(p,v,\xi)\\
				\text{s.t. }\left\{ \begin{array}{l l}
					p\in [0,1] &  \\                
					v\text{ solves } \left\{\begin{array}{l}
						\displaystyle\min_v \E\left(f\left(v,\xi\right)\right) \\
						\text{ s.t. } v\in [0,1].
					\end{array}\right.
				\end{array} \right.
			\end{array}
			&
			\begin{array}{l}
				\displaystyle\min_{p,v}\,\,  \theta_{-}(p,v,\xi)\\
				\text{s.t. }\left\{ \begin{array}{l l}
					p\in [0,1] &  \\                
					v\text{ solves } \left\{\begin{array}{l}
						\displaystyle\min_v \E\left(f\left(v,\xi\right)\right) \\
						\text{ s.t. } v\in [0,1].
					\end{array}\right.
				\end{array} \right.
			\end{array}\\
			&\\
			\hline
			\text{Plus Case}&\text{Minus Case}\\
			\hline
		\end{array}
		\]
		It is easy to see that the optimal solutions for the leader and the follower are very similar. Observe too that the decision of the leader has no influence on the follower's decision, but the objective of the follower is directly included in the leader's objective. 
		
		\begin{table}[!ht]
			\begin{center}
				\begin{minipage}{0.45\textwidth}
					\centering
					\begin{tabular}{c|c|c|}
						& Plus Case & Minus Case \\
						\hline
						$WS$      & $0.25$& $-0.25$ \\
						\hline
						$SWS$     & $0$   & $0$ \\
						\hline
					\end{tabular}
					\vspace{0.3cm}
					\caption{Values of STO, WS and SWS}\label{bound_values}
				\end{minipage}
			\end{center}
		\end{table}
		
		The set of optimal decisions is easy to determine, and  $WS$ and $SWS$ are displayed in Table \ref{bound_values}.
		$\hfill\Diamond$
	\end{example}
	
	The intuition behind this example is simple. On the one hand, the Plus Case is collaborative: the leader wants to collaborate with the follower, since the leader is losing what the follower loses as well. On the other hand, the Minus Case is adversarial: the leader is against the follower, since what the follower loses translates in gains for the leader. However, in practical situations, the collaboration or competition between the leader and the follower might not be so clear.

	
	\section{Reformulation to Single Bilinear Optimization \label{sec:Reformulation}}
	
	In this section, we go back to the model presented in Section \ref{sec:Model}, and we \GL{focus on} how to compute WS and SWS for Problem \eqref{eq:Leaders-Problem-WS}, and then obtaining the EVSI. In what follows, we identify the scenario set $\{\omega_1,\ldots,\omega_m\}$ with the set of indexes $K=\{1,\ldots,m\}$. \medskip
	
	In order to \GL{computing} the EVSI, our approach is to follow a Monte-Carlo estimation: first, we consider a sample $(y^1,d_0^1),\ldots,(y^T,d_0^T)$ of the random parameters $(y,d_0)$, accordingly to the leader's distributions $D_L$. Then, to compute the empirical expectations of the value functions $\psi$ given by \eqref{eq:psi} for the WS, and $\varphi$ given by \eqref{eq:Def-ParametricProblem} for the SWS. Finally, we compute the EVSI as
	\begin{equation}\label{eq:EVSI-sampled}
		EVSI = \frac{1}{T}\sum_{t=1}^T \varphi(y^t,d_0^t)-\frac{1}{T}\sum_{t=1}^T \psi(y^t,d_0^t).
	\end{equation}
	Thus, our problem is reduced to compute the value functions $\psi(y^t,d_0^t)$ and $\varphi(y^t,d_0^t)$ for each sample $(y^t,d_0^t)$. Our technique is, for both values, to reformulate the corresponding bilevel programming problems into single level bilinear problems. In both cases, we replace the corresponding follower's problem by its Karush-Kuhn-Tucker (KKT) conditions and consider the associated multipliers as new variables. This approach, known as the Mathematical Programming with Complementarity Constraints (MPCC) reformulation, is quite popular in the literature and can be applied whenever the follower's problem satisfies a constraint qualification (see, e.g., \cite[Chapter 3]{DempeEtAl2015} and the references therein). \medskip

	In this section, we show that the reformulations we obtain \GL{throughout} this technique have two main properties: firstly, they preserve global solutions in the sense that a pair $(p,v)$ of leader-follower decision variables is a global solution of a bilevel program if and only if there exists a multiplier $u$ such that $(x,y,u)$ is a global solution of the MPCC reformulation; and secondly, the MPCC reformulations can be rewritten as mixed-integer bilinear problems.\medskip
	
	\DSr{To facilitate the following developments, we identify $\R^{n(n-1)}$, which is the space of decision variables of the follower, with the subspace $V$ of $n\times n$ square off-diagonal matrices (i.e. with $0$-entries in the diagonal).}\medskip
	
	Before studying the single-level reformulations for $WS$ and $SWS$, we will study the regularity properties of the feasible set of Problem $F_m(p)$, defined in  \eqref{eq:FollowerProblem-DiscreteDistribution}. To do so, we consider the following key lemma. 
	
	\begin{lemma}\label{lemma:LICQ} Assume that the initial allocation vector $x_0$ is strictly positive (i.e. $x_{0,i}>0$ for each $i\in \{1,\ldots, n\}$). Then, the feasible set of the followers' problem, which is given by
		\[
		\left\{ v\in V\ :\ \begin{array}{c}
			v\geq 0  \\
			\displaystyle \sum_{j\neq i} v_{ij} \leq x_{0,i},\quad\forall i\in I 
		\end{array} \right\}
		\]
		satisfies Slater's CQ, and it satisfies (LICQ) at every point.
	\end{lemma}
	\begin{proof}  For every $i,j\in\{1,\ldots,n\}$ with $i\neq j$, we denote by $e_{ij}$ as the $n\times n$ matrix given by
		\begin{equation}\label{eq:BasisRepresentation}
			e_{ij}(a,b) = \begin{cases}
				1\quad& \text{ if }a=i\text{ and }b=j,\\
				0&\text{ otherwise,}
			\end{cases}
		\end{equation}
		and we denote by $e_{i\bullet} = \sum_{j\ :\ j\neq i} e_{ij}$, which is the $n\times n$ matrix with $1$-entries in the $i$th row (except for the entry $(i,i)$), and $0$ otherwise. Similarly, we set $e_{\bullet j} = \sum_{i\ :\ i\neq j} e_{ij}$ which is the $n\times n$ matrix with $1$-entries in the $j$th column (except for the entry $(j,j)$), and $0$ otherwise. Now, we define $x_{0,\min} = \min\{x_{0,1},\ldots,x_{0,n}\}$. We claim that $\bar{v}\in V$ defined by
		\[
		\bar{v}_{ij} = \frac{x_{0,\min}}{2(n-1)}
		\]
		is a Slater point. In fact, it is clear that $\bar{v}_{ij}>0$ for every $i\neq j$, and furthermore,
		\[
		\sum_{j\neq i} \bar{v}_{ij} = \sum_{j\neq i} \frac{x_{0,\min}}{2(n-1)} = (n-1)\frac{x_{0,\min}}{2(n-1)} = \frac{x_{0,\min}}{2} < x_{0,i},\quad \forall i\in I.
		\]
		Thus, the claim is verified and this finishes the first part of the proof. Now, let us show that (LICQ) is verified at every point. For every $i\in I$, and every $j\neq i$, let $g_{ij}(v) = -v_{ij}$ and let $h_i(v) = \sum_{j\neq i} v_{ij} - x_{0,i}$. Then, this set can be written as
		\[
		\left\{ v\in V\ :\ \begin{array}{ll}
			g_{ij}(v)\leq 0,&\quad\forall i\neq j  \\
			h_i(v)\leq 0,&\quad\forall i\in I 
		\end{array} \right\}.
		\]
		Now, suppose that there exists $v^*$ in this set, not satisfying (LICQ). It is not hard to see that $\nabla g_{ij}(v^*) = -e_{ij}$ and $\nabla h_i(v^*) = e_{i\bullet}$. Thus, since $\{\nabla g_{ij}(v^*)\ :\ i\neq j\}$ is linearly independent, there must be $i_0\in I$ such that $h_{i_0}(v^*)=0$, and such that $\nabla h_{i_0}(v^*) = e_{i_0\bullet}$ is a linear combination of the gradients of the other active constraints. However, this is only possible if $g_{i_0j}$ is active at $v^*$ for every $j\neq i_0$, which would mean that
		\[
		\sum_{j\neq i_0} v_{i_0j} = x_{0,i_0}\quad\text{ and }\quad v_{i_0j} = 0,\,\quad \forall j\neq i_0,
		\]
		which is a contradiction since $x_{0,i_0}\neq 0$. This finishes the proof.
	\end{proof}
	
	\DSr{The following two subsections contain the reformulations of the WS \GL{and} SWS problems. The reader should consider them independent, since some notation might be overlapped.}
	
	
	\subsection{Reformulation of Wait-and-See}
	Recall that we want to solve Problem \eqref{eq:psi}, which in this context is given by 
	
	\begin{equation}\label{eq:psi-ws}
		\psi(y,d_0) := \left\{ \begin{array}{l}
			\displaystyle\min_{p,v}\,\,\sum_{i=1}^n (1-c) p_i\max(-x_i-y_i,-d_i)\\
			\text{s.t. }\left\{ \begin{array}{l l}
				p_i \in [ p_{i,\min}, p_{i,\max} ],\quad \forall i\in I\\
				v \mbox{ solves } F_m(p).
			\end{array} \right.
		\end{array}\right.
	\end{equation}
	
	Based on (\ref{eq:ScenariosDecomposition}), for each scenario $k\in K$ and each location $i\in I$, we set (recalling that $y_i\sim U(0,\bar{y})$) a function $\phi_{i,k}$ as follows:
	\begin{equation}\label{eq:Def-phi}
		\begin{aligned}
			\phi_{i,k}(x_i) &= \E_{y_i\sim U}\left[ \max(d_{i,k} - x_i-y_i,0) \right]\\
			&= \frac{1}{\bar{y}} \int_0^{\bar{y}} \max(d_{i,k}-x_i-z,0) dz \\
			&= \begin{cases} 0 & \mbox{ if }d_{i,k}-x_i\leq 0, \\ \frac{(d_{i,k}-x_i)^2}{2\bar{y}} & \mbox{ if } 0\leq d_{i,k}-x_i\leq \bar{y}, \\ (d_{i,k}-x_i) -  \frac{\bar{y}}{2} & \mbox{ if }d_{i,k}-x_i \geq \bar{y}.\end{cases}
		\end{aligned}
	\end{equation}
	Then, it is not hard to see that, considering for $d_0$ a discrete distribution $\rho$ given by 
	\[
	\rho(A) = \sum_{k=1}^m 1\!\!1_A(\omega_k)\rho_k,
	\]
	the follower's problem $F_m(p)$ can be written as
	\begin{equation}\label{eq:Fm(p)-finalFormulation}
		F_m(p) := \begin{cases}
			\displaystyle \min_{\GL{v}} & c \displaystyle \sum_{i=1}^n \sum_{k=1}^m p_i(\phi_{k,i}(x_i) - d_{i,k})\cdot \rho_k + \sum_{i\neq j} \alpha_{ij}v_{ij}\\
			\mbox{s.t.}\quad&\begin{cases}
				-v\leq 0\\
				\displaystyle \sum_{j\neq i}v_{ij} - x_{0,i}\leq 0,\quad \forall i\in I.
			\end{cases}
		\end{cases}
	\end{equation}
	
	Then, we can state the following proposition.
	
	\begin{proposition}\label{thm:WS} For any given value of the random vector $\xi = (y,d_0)$, the Wait-and-See problem of Definition \ref{def:EVPI}, associated \GL{with} the leader's problem \eqref{eq:Leaders-Problem-WS}, is equivalent (in the sense of local and global solutions) to its MPCC reformulation given by
		\begin{equation}\label{eq:MPCC-WS}
			\begin{array}{cl}
				\displaystyle\max_{p,v,\lambda,\gamma} & \displaystyle \sum_{i=1}^n (1-c) p_i\cdot \min(x_i+y_i,d_i)\\
				& \text{\emph{s.t.}} \begin{cases}
					p_i \in [ p_{i,\min}, p_{i,\max} ],\quad \forall i\in I\\
					\displaystyle \sum_{j\neq i}v_{ij} - x_{0,i}\leq 0,\quad \forall i\in I\\
					\displaystyle\sum_{k=1}^m(p_i\beta_{i,k} - p_j\beta_{j,k}) + \alpha_{ij}- \lambda_{ij} + \gamma_{i} = 0,\quad \forall i\neq j\in I\\
					\displaystyle \gamma_i(\sum_{j\neq i}v_{ij} - x_{0,i}) = 0,\quad\forall i\in I\\
					\lambda_{ij}v_{ij} = 0,\quad \forall i\neq j\in I\\
					v\geq 0,\gamma,\lambda \geq 0,
				\end{cases}
			\end{array}
		\end{equation}
		
		where the coefficients $\{\beta_{i,k}\ :\ i\in I, k\in K\}$ are given by
		
		\begin{equation}\label{eq:CoefBeta-EVPI}
			\beta_{i,k} := \begin{cases}
				0 & \mbox{ if }d_{i,k}-x_i\leq 0 \\ c\rho_k\frac{x_i-d_{i,k}}{\bar{y}} & \mbox{ if } 0\leq d_{i,k}-x_i\leq \bar{y} \\
				-c\rho_k & \mbox{ if }d_{i,k}-x_i \geq \bar{y}.
			\end{cases} 
		\end{equation}
		
		Furthermore, the multipliers $\gamma = (\gamma_i\ :\ i\in I)$ and $\lambda = (\lambda_{ij}\ :\ i\neq j\in I)$ verify
		
		\begin{equation}\label{eq:BoundsMultipliers-PerfectInfo}
			0\leq \gamma_i \leq 2mp_{\max}\quad\mbox{ and }\quad 0\leq \lambda_{ij}\leq \DSr{4mp_{\max}+\alpha_{ij}},
		\end{equation}
		
		where $ p_{\max} = \displaystyle \max_{\GL{i\in I}} \{p_{i,\max}\}$.
	\end{proposition}
	
	\begin{proof}
		The equivalence between the bilevel problem \eqref{eq:Leaders-Problem-WS} and the MPCC reformulation follows from Lemma \ref{lemma:LICQ}. Indeed, since Slater's CQ is verified, and (LICQ) implies Constant Rank CQ (see the definition in, e.g., \cite{DempeDutta2012}), the desired conclusion follows from \cite[Theorem 3.2 and Corollary 3.3]{DempeDutta2012}. Thus, it is enough to show that problem \eqref{eq:MPCC-WS} coincides with the MPCC reformulation of problem \eqref{eq:Leaders-Problem-WS}.\medskip
		
		Based on \eqref{eq:Def-phi} we can compute the partial derivatives of $\phi_{i,k}$ with respect to $x_i$ as
		\[
		\frac{\partial}{\partial x_i} \phi_{i,k} = \begin{cases}
			0 & \mbox{ if }d_{i,k}-x_i\leq 0 \\ \frac{x_i-d_i(\omega_k)}{\bar{y}} & \mbox{ if } 0\leq d_{i,k}-x_i\leq \bar{y} \\ -1 & \mbox{ if }d_{i,k}-x_i \geq \bar{y}
		\end{cases} 
		\]
		
		Let $f(v)$ be the objective function for $F_m(p)$. Recalling that that $x_i = x_i(v)$ given by \eqref{eq:Kirchoff}, by a mild application of the chain rule, considering the definition of $x_r$ in Equation \eqref{eq:Kirchoff} we can compute its partial derivative as
		\begin{align*}
			\frac{\partial f}{\partial v_{ij}}(v) &= \frac{\partial}{\partial v_{ij}}\left( c\sum_{r=1}^n\sum_{k=1}^m p_r\phi_{r,k}(x_r)\rho_k - c\sum_{r=1}^n\sum_{k=1}^m p_rd_r(\omega_k)\rho_k + \sum_{r\neq s} \alpha_{rs}v_{rs} \right)\\ 
			&=\frac{\partial}{\partial v_{ij}}\left( c\sum_{k=1}^m p_i\phi_{i,k}(x_i)\rho_k+ c\sum_{k=1}^m p_j\phi_{j,k}(x_j)\rho_k +\alpha_{ij}v_{ij}  \right)\\
			&= c\sum_{k=1}^m (p_i \frac{\partial}{\partial  x_i}\phi_{i,k}(x_i)\rho_k - p_j \frac{\partial}{\partial  x_j}\phi_{j,k}(x_j)\rho_k) +\alpha_{ij}\\
			&= \sum_{k=1}^m(p_i\beta_{i,k} - p_j\beta_{j,k}) + \alpha_{ij}, 
		\end{align*}
		where the coefficients $\{\beta_{i,k}\ :\ i\in I, k\in J\}$ are defined as in \eqref{eq:CoefBeta-EVPI}. Hence, the KKT equations for $F_m(p)$ have the form 
		\begin{align*}
			\sum_{k=1}^m(p_i\beta_{i,k} - p_j\beta_{j,k}) + \alpha_{ij}- \lambda_{ij} + \gamma_{i} = 0,\quad \forall i,j\in\{1,\ldots,n\},\ i\neq j
		\end{align*}
		with the complementarity constraints 
		\begin{equation}\label{eq:In-Proof:complementarity}
			\begin{aligned}
				&\lambda_{ij}v_{ij} = 0,\quad \forall i,j\in\{1,\ldots,n\},\ i\neq j.\\
				&\gamma_i\left(\sum_{j\neq i} v_{ij} - x_0{_i}\right) = 0,\quad \forall i\in\{1,\ldots,n\}.
			\end{aligned}
		\end{equation}
		Putting all together, we get that the MPCC reformulation of \eqref{eq:Leaders-Problem-WS} is given by \eqref{eq:MPCC-WS}. \medskip
		
		Lastly, we compute the multiplier bounds. Fix a feasible price vector $p$ and let $v^*$ be an optimal point of $F_m(p)$. Let $(\lambda,\gamma,\beta)$ be a feasible tuple of multipliers for problem \eqref{eq:MPCC-WS}. Then, for the $i$th coordinate, we have that 
		\[
		\sum_{k=1}^m(p_i\beta_{i,k} - p_j\beta_{j,k}) + \alpha_{ij}- \lambda_{ij} + \gamma_{i} = 0,\quad \forall j\neq i.
		\] 
		We have two possible scenarios:
		\begin{itemize}
			\item If $\gamma_i=0$, then 
			\[
			\DSr{\lambda_{ij} = \sum_{k=1}^m \left(p_i\beta_{i,k}-p_j\beta_{j,k}\right)+\alpha_{ij} \leq 2mp_{\max} + \alpha_{ij},\quad \forall j\in\{1,\ldots,n\},}
			\]
			\DSr{where the inequality follows directly by noting that the $\beta$ coefficients are always less than 1.}
			\item If $\gamma_i\neq 0$, then the second complementary equation of \eqref{eq:In-Proof:complementarity} implies that $x_0{_i} = \sum_{j\neq i} v^*_{ij}$. Since $x_0{_i}\neq 0$, this also implies that there exists a value of $j$ such that $v^*_{ij}\neq 0$, in which case $\lambda_{ij} = 0$ using the first complementary equation of \eqref{eq:In-Proof:complementarity}. Hence, we conclude that
			\[
			\begin{array}{c}
				\sum_{k=1}^m \left(\beta_{i,k}-\beta_{j,k}\right) + \alpha_{ij} + \gamma_i = 0\\
				\implies\\
				\gamma_i = \sum_{k=1}^m \left(\beta_{j,k}-\beta_{i,k}\right) - \alpha_{ij} \leq 2mp_{\max},
			\end{array}
			\] 
			and so we can compute  
			\[
			\lambda_{ij} = \sum_{k=1}^m \left(\beta_{i,k}-\beta_{j,k}\right) + \gamma_i + \alpha_{ij} \leq 4mp_{\max} + \alpha_{ij},\quad \forall j\in\{1,\ldots,n\}.
			\]
		\end{itemize}
		Regardless the case, we conclude that 
		\[
		\lambda_{ij}\in[0,4mp_{\max}+\alpha_{ij}],\quad \gamma_i\in[0,2mp_{\max}],
		\]
		finishing our proof.
	\end{proof}
	
	This result allows us to estimate the Wait-and-See value by sampling $\xi = (y,d_0)$ and solving \eqref{eq:MPCC-WS}. To do so, we will follow the classic big-M strategy, which seems to be first introduced in the context of bilevel optimization in \cite{Fortuny-AmatMcCarl1981}. Even though computing a sufficiently large $M$ is hard in general \cite{KleinertLabbePleinSchmidt2020}, the above proposition has already provided the needed bounds in (\ref{eq:BoundsMultipliers-PerfectInfo}). Hence, defining $M = 4mp_{\max}$, we proceed as follows:
	\begin{enumerate}
		\item For each pair $i\neq j\in I$, we introduce a boolean variable $z_{ij}\in \{0,1\}$ and replace the constraint $\lambda_{ij}v_{ij} = 0$ by
		\begin{equation}
			\begin{aligned}
				&-M z_{ij}\leq \lambda_{ij} \leq M z_{ij}\\
				&-M (1-z_{ij})\leq v_{ij}\leq M (1-z_{ij})
			\end{aligned}
		\end{equation}
		\item For each $i\in I$, we introduce a boolean variable $w_{i}\in \{0,1\}$ and replace the constraint $\gamma_i(\sum_{j\neq i}v_{ij} - x_{0,i}) = 0$ by
		\begin{equation}
			\begin{aligned}
				&-M w_i\leq \gamma_i \leq M w_i\\
				&-M (1-w_i)\leq \sum_{j\neq i} v_{ij} - x_{0,i}\leq M (1-w_i)
			\end{aligned}
		\end{equation}
	\end{enumerate}
	
	The last numerical consideration involves the additional constraints which are used to tackle the term $\beta_{i,k}$, which are given piecewise linear functions. We first define a constant
	\begin{equation*}
		C \geq \max\{ d_{0,i}(\omega_k)\ :\ i\in I, k\in J  \} + N_0,
	\end{equation*}
	which is an upper bound of $\mid\! d_i(\omega_k) - x_i\!\mid$ for every $i\in I$ and every scenario $\omega_k$. Then, for each term $\beta_{i,k}$, we define three integer variables $a_{i,k},b_{i,k},c_{i,k}\in \{0,1\}$, three continuous variables $r_{i,k}$, $s_{i,k}$, $t_{i,k}$, and we replace \eqref{eq:CoefBeta-EVPI} by the set of constraints
	\begin{equation}
		\begin{cases}
			a_{i,k}+b_{i,k}+c_{i,k} = 1\\
			d_{i,k}-x_i = r_{i,k}+s_{i,k}+t_{i,k}\\
			-Ca_{i,k} \leq r_{i,k}\leq 0 \\
			0 \leq  s_{i,k} \leq \bar{y}b_{i,k} \\
			\bar{y}c_{i,k} \leq t_{i,k}\leq Cc_{i,k} \\
			\beta_{i,k} = -c\rho_k\left(\frac{s_{i,k}}{\bar{y}} + c_{i,k}\right)
		\end{cases}
	\end{equation}
	
	The above replacement works as follows: 
	
	\begin{itemize}
		\item If $a_{i,k} = 1$, then $b_{i,k}=c_{i,k}=s_{i,k}=t_{i,k} = 0$. Hence, $d_{i,k} - x_i = r_{i,k} \leq 0$, and $\beta_{i,k} = 0$.
		\item If $b_{i,k} = 1$, then $a_{i,k}=c_{i,k}=r_{i,k}=t_{i,k} = 0$. Hence, $d_{i,k} - x_i = s_{i,k} \in [0,\bar{y}]$, and $\beta_{i,k} = -c\rho_k\frac{s_{i,k}}{\bar{y}} = c\rho_k\frac{\left(x_i-d_{i,k}\right)}{\bar{y}}$.
		\item If $c_{i,k} = 1$, then $a_{i,k}=b_{i,k}=r_{i,k}=s_{i,k} = 0$. Hence, $d_{i,k} - x_i = t_{i,k} \geq \bar{y}$, and $\beta_{i,k}  = -c\rho_k$.
	\end{itemize}
	
	The final problem we solve for each sample of $\xi = (y,d_0)$ is then given by
	\begin{equation}\label{eq:Problem-WS-MIP}\small
		\begin{array}{cl}
			\displaystyle\max_{\tiny\GL{\begin{array}{c} p,v,\lambda,\gamma,\beta,\\
						a,b,c,r,s,t\end{array}}} & \displaystyle \sum_{i=1}^n (1-c) p_i\cdot \min(x_i+y_i,d_i)\\
			\mbox{s.t.}	& \begin{cases}
				p_i \in [ p_{i,\min}, p_{i,\max} ],\quad \forall i\in I\\
				\displaystyle \sum_{k=1}^m(p_i\beta_{i,k} - p_j\beta_{j,k}) + \alpha_{ij}- \lambda_{ij} + \gamma_{i} = 0,\quad \forall i\neq j\in I,\\
				-M z_{ij}\leq \lambda_{ij} \leq M z_{ij}\\
				-M (1-z_{ij})\leq v_{ij}\leq M (1-z_{ij})\\
				-M w_i\leq \gamma_i \leq M w_i\\
				\displaystyle -M (1-w_i)\leq \sum_{j\neq i} v_{ij} - x_0{_i}\leq M (1-w_i)\\
				a_{i,k},b_{i,k}, c_{i,k}\in \{0,1\},\quad \forall i\in I,\forall k\in K\\ a_{i,k}+b_{i,k}+c_{i,k} = 1 ,\quad \forall i\in I,\forall k\in K\\
				d_{i,k}-x_i = r_{i,k}+s_{i,k}+t_{i,k} ,\quad \forall i\in I,\forall k\in K\\
				-Ca_{i,k} \leq r_{i,k}\leq 0 ,\quad \forall i\in I,\forall k\in J\\
				0 \leq  s_{i,k} \leq \bar{y}b_{i,k} ,\quad \forall i\in I,\forall k\in K\\
				\bar{y}c_{i,k} \leq t_{i,k}\leq Cc_{i,k} ,\quad \forall i\in I,\forall k\in K\\
				\beta_{i,k} = -c\rho_k\left(\frac{s_{i,k}}{\bar{y}} + c_{i,k}\right) ,\quad \forall i\in I,\forall k\in K.
			\end{cases}
		\end{array}
	\end{equation}
	
	
	\subsection{Reformulation of Shared-Wait-and-See}
	
	When the leader shares the value of  $\xi =(y,d_0)$ with the follower, we must consider this information in the objective function. Since prices are also parameters, the demand vector $d$ becomes fixed (given by \eqref{eq:Def-Demand}) and known by the follower. The leader then must solve
	\begin{equation}\label{eq:Leaders-Problem-SWS}
		\varphi(y,d_0) := \begin{cases}
			\displaystyle \max_{\GL{p,v}} &\sum_{i=1}^n (1-c) p_i\cdot \min(x_i+y_i,d_i)\\
			\mbox{s.t.}\quad&\begin{cases}
				p_i \in [ p_{i,\min}, p_{i,\max} ],\quad \forall i\in I\\
				v \mbox{ solves } F(p).
			\end{cases}
		\end{cases}   
	\end{equation}
	where $F(p)$ is given by 
	\begin{equation}\label{eq:Final-F(p)-shared}
		F(p) := \begin{cases}
			\displaystyle \min_{\GL{v}} & \sum_{i=1}^n cp_i\max(-x_i-y_i,-d_i) + \sum_{i\neq j} \alpha_{ij}v_{ij}\\
			\mbox{s.t.}\quad&\begin{cases}
				-v\leq 0\\
				\sum_{j\neq i}v_{ij} - x_{0,i}\leq 0,\quad \forall i\in I
			\end{cases}
		\end{cases}
	\end{equation}
	
	In this scenario, we can state the following proposition.
	
	\begin{proposition}\label{thm:SWS} For any given value of the random vector $\xi = (y,d_0)$, the Shared-Wait-and-See problem of Definition \ref{def:EVSI}, associated \GL{with} the leader's problem (\ref{eq:Leaders-Problem-WS}), is equivalent (in the sense of local and global solutions) to its MPCC reformulation given by
		\begin{equation}\label{eq:MPCC-SWS}
			\begin{array}{cl}
				\displaystyle\max_{\GL{p,v,\lambda,\gamma}} & \displaystyle \sum_{i=1}^n (1-c) p_i\cdot \min(x_i+y_i,d_i)\\
				\mbox{\emph{s.t.}}&\begin{cases}
					p_i \in [ p_{i,\min}, p_{i,\max} ],\quad \forall i\in I\\
					\displaystyle\sum_{j\neq i}v_{ij} - x_{0,i}\leq 0,\quad \forall i\in I\\
					c(\beta_i - \beta_j) + \alpha_{ij} - \lambda_{ij} + \gamma_{i} = 0,\quad \forall i\neq j \\
					\displaystyle\gamma_i(\sum_{j\neq i}v_{ij} - x_{0,i}) = 0,\quad\forall i\in I\\
					\lambda_{ij}v_{ij} = 0,\quad \forall i\neq j\in I\\
					v\geq 0, \gamma,\lambda \geq 0,
				\end{cases}
			\end{array}
		\end{equation}
		where the variables $\{\beta_{i}\ :\ i\in I\}$ verify that
		\begin{equation}\label{eq:CoefBeta-EVSI}
			\beta_i \in \begin{cases} \{0\} & \mbox{ if }d_i-x_i < y_i \\ \{p_i\} & \mbox{ if } d_i-x_i > y_i \\  [0,p_i] & \mbox{ if }d_i-x_i=y_i\end{cases}. 
		\end{equation}
		
		Furthermore, the multipliers $\gamma = (\gamma_i\ :\ i\in I)$ and $\lambda = (\lambda_{ij}\ :\ i\neq j\in I)$ verify that
		\begin{equation}\label{eq:BoundsMultipliers-PerfectInfo-2}
			0\leq \gamma_i \leq 2p_{\max}\quad\mbox{ and }\quad 0\leq \DSr{\lambda_{ij}\leq 4p_{\max}+\alpha_{ij}},
		\end{equation}
		where $p_{\max} = \displaystyle \max_{\GL{i\in I}} \{p_{i,\max}\}$.
	\end{proposition}
	
	\begin{proof}
		The equivalence between the bilevel problem \eqref{eq:Leaders-Problem-SWS} and the MPCC reformulation follows as in the proof of Proposition \ref{thm:WS}. Thus, it is enough to show that problem \eqref{eq:MPCC-SWS} coincides with the MPCC reformulation of problem \eqref{eq:Leaders-Problem-SWS}.\medskip
		
		Fix a pair of multipliers $(\lambda,\gamma)$. As the objective function for the follower is non-differentiable this time, the Fermat condition within the Karush-Kuhn-Tucker equations is given by the inclusion $0\in \partial \mathcal{L}(v)$, where 
		\[
		\mathcal{L}(v) = \sum_{i=1}^n cp_i\max\{-x_i-y_i,-d_i\} + \sum_{i\neq j} \alpha_{ij}v_{ij} - \langle\lambda,v\rangle +\sum_{i\in I}\gamma_i(\sum_{j\neq i}v_{ij} - x_{0,i}).
		\]
		
		Furthermore, as all the involved functions in this formula are convex and continuous we can compute the required subdifferential as a sum of separated subdifferentials (see, e.g., \cite{RockafellarWets1998Variational}). If we call $\Gamma(t) = \partial (\max\{0,\cdot\})(t)$, it is clear that 
		\[
		\Gamma(t) = \begin{cases} \{0\} & \mbox{ if }t<0 \\
			\{1\} & \mbox{ if }t>0 \\ [0,1] & \mbox{ if }t=0
		\end{cases}.
		\]
		Therefore, Recalling that that $x_i = x_i(v)$ given by \eqref{eq:Kirchoff}, by the convex subdifferential chain rule (see, e.g., \cite[Chapter 16]{BauschkeCombettes2017}), 
		\begin{align*}
			\partial \left(\sum_{i=1}^n cp_i\max\{-x_i-y_i,-d_i\}\right) &= \sum_{i=1}^n cp_i \partial \left(\max\{0,-x_i-y_i+d_i\}\right) \\
			&= \sum_{i=1}^n cp_i \Gamma(-x_i-y_i+d_i)\nabla_v(-x_i) \\
			&= \sum_{i=1}^n cp_i \Gamma(-x_i-y_i+d_i)(e_{i\bullet} - e_{\bullet i})
		\end{align*}
		With this formula in mind, the inclusion $0\in \partial \mathcal{L}(v)$ is equivalent to the existence of a vector $\beta \in \R^{n}$ such that $\beta_i\in p_i\Gamma (-x_i-y_i+d_i)$ for every $i\in I$, and such that
		\begin{align*}
			0 = \sum_{i=1}^n c \beta_i(e_{i\bullet} - e_{\bullet i}) + \nabla \left(\sum_{i\neq j} \alpha_{ij}v_{ij} - \langle\lambda,v\rangle +\sum_{i\in I}\gamma_i (\sum_{j\neq i}v_{ij} - x_{0,i})\right).
		\end{align*}
		The above vector equation can be equivalently written as the set of equations
		\begin{align*}
			0 = c(\beta_i - \beta_j) + \alpha_{ij} - \lambda_{ij} + \gamma_{i},\quad\forall i\neq j.
		\end{align*}
		
		where $\{\beta_i\ :\ i\in I\}$ are new variables verifying the inclusion \eqref{eq:CoefBeta-EVSI}.\medskip
		
		The complementary equations are still the same as in \eqref{eq:In-Proof:complementarity}. Thus, putting all together, we conclude that the MPCC reformulation of \eqref{eq:Leaders-Problem-SWS} is indeed given by \eqref{eq:MPCC-SWS}. Finally, similar to the Wait-and-See case, we can prove the following:
		\begin{itemize}
			\item If $\gamma_i=0$, then \[\lambda_{ij} = c\beta_i - c\beta_j+\alpha_{ij} \leq \DSr{2p_{\max} + \alpha_{ij}},\quad \forall j\in\{1,\ldots,n\}.\]
			\item If $\gamma_i\neq 0$, then the second complementary equation implies that $x_0{_i} = \sum_{j\neq i} v^*_{ij}$. Since $x_0{_i}\neq 0$, this also implies that there exists a value of $j$ such that $v^*_{ij}\neq 0$, in which case $\lambda_{ij} = 0$ using the first complementary equation. Hence, we conclude that
			\[
			c\beta_i-c\beta_j + \alpha_{ij} + \gamma_i = 0 \implies \gamma_i = c\beta_j-c\beta_i - \alpha_{ij} \leq 2p_{\max},\] and so we can compute  \[\lambda_{ij} = c\beta_i-c\beta_j + \gamma_i + \alpha_{ij} \leq 4p_{\max},\quad \forall j\in\{1,\ldots,n\}.
			\]
			
		\end{itemize}
		Regardless the case, we conclude that 
		\[
		\lambda_{ij}\in[0,4p_{\max}+\alpha_{ij}],\quad \gamma_i\in[0,2p_{\max}],
		\]
		finishing our proof.
	\end{proof}
	
	This proposition allows us to replicate the big-M strategy used with the Wait-and-See value, regarding the complementarity constraints. In the Shared-Wait-and-See case, one last numerical consideration involves the additional constraints used to tackle the $\beta_i$ terms. We define now \[C=p_{\max}(\max\{ d_{0,i}(\omega_k)\ :\ i\in I, k\in J  \} + N_0)\] and two additional boolean variables,  
	\[
	s_i,t_i\in\{0,1\}\ ,\quad s_i+t_i\leq 1
	\]
	so at most, one of them gets the value $1$. Then, we add the following constraints:
	\[
	\beta_i\leq p_{i,\max}(1-s_i)\ ,\quad x_i+y_i-d_i \geq -C(1-s_i)
	\]
	\[
	p_i- \beta_i\leq p_{i,\max}(1-t_i)\ ,\quad x_i+y_i-d_i \leq C(1-t_i)
	\]
	\[
	x_i+y_i-d_i\leq C(s_i+t_i)\ ,\quad x_i+y_i-d_i \geq -C(s_i+t_i).
	\]
	
	Here we have three feasible scenarios:
	
	\begin{itemize}
		\item If $s_i = 1$, then $t_i = 0$, the first set of equations leads to $\beta_i=0$, and the other ones leave $x_i+y_i-d_i$ able to get positive values.
		\item If $t_i = 1$, then $s_i = 0$, the second set of equations leads to $\beta_i = p_i$, and the other ones leave $x_i+y_i-d_i$ able to get negative values.
		\item If $s_i = t_i = 0$, the third set of equations leads to $x_i+y_i-d_i = 0$, and the other ones leave $\beta_i\in [0,p_{i,\max}]$.
	\end{itemize}
	
	The final problem we solve for each $\zeta = (y,d_0)$ is then given by
	\begin{equation}\label{eq:Reformulation-SWS-MIP}\small
		\begin{array}{cl}
			\displaystyle \max_{\GL{p,v,\lambda,\gamma,s,t,w,z,\beta}} & \displaystyle \sum_{i=1}^n (1-c) p_i\cdot \min(x_i+y_i,d_i)\\ 
			\mbox{s.t.}& \begin{cases}
				p_i \in [ p_{i,\min}, p_{i,\max} ],\quad \forall i\in I\\
				c(-\beta_i + \beta_j) + \alpha_{ij} - \lambda_{ij} + \gamma_{i} = 0,\quad \forall i\in I,\ \forall j\neq i \\
				-M z_{ij}\leq \lambda_{ij} \leq M z_{ij}\\
				-M (1-z_{ij})\leq v_{ij}\leq M (1-z_{ij})\\
				-M w_i\leq \gamma_i \leq M w_i\\
				\displaystyle -M (1-w_i)\leq \sum_{j\neq i} v_{ij} - x_0{_i}\leq M (1-w_i)\\
				\displaystyle x_{i} = x_{0,i} + \sum_{j\neq i} v_{ji} - \sum_{k\neq i} v_{ik},\quad\forall i\in I \\
				\beta_i\leq p_{i,\max}(1-s_i), \quad \forall i\in I\\ 
				x_i+y_i-d_i \geq -C(1-s_i),\quad \forall i\in I\\
				p_i- \beta_i\leq p_{i,\max}(1-t_i),\quad \forall i\in I\\
				x_i+y_i-d_i \leq C(1-t_i),\quad \forall i\in I\\
				x_i+y_i-d_i\leq C(s_i+t_i)\quad \forall i\in I\\
				x_i+y_i-d_i \geq -C(s_i+t_i),\quad \forall i\in I \\
				s_i+t_i\leq 1,\quad \forall i\in I \\
				z_{ij}\in\{0,1\},\ w_i,s_i,t_i\in\{0,1\},\quad \forall i\in I
			\end{cases}
		\end{array}
	\end{equation}
	
	\begin{remark}\label{rem:bilinearity}
		It is worth to notice that both MPCC reformulations \eqref{eq:Problem-WS-MIP} and \eqref{eq:Reformulation-SWS-MIP} are mixed-integer bilinear programming problems. \GL{The bilinearity comes from two sources: 1) the (implicit) product $p_i\cdot (v_{ij}-v_{ji})$ in the objective function (becoming non-concave), which is common for both WS and SWS problems; and 2) the product $p_i\beta_{i,j}$ appearing in the second line of constraints of the WS final reformulation. This second source of bilinearity is much harder to handle since it is formed by polynomial equality constraints.} While in general even the simplest bilevel programming problems are NP-hard \cite{Ben-Ayed_Computational_1990},
		mixed-integer bilinear problems can be handled efficiently by some commercial solvers, like \texttt{Gurobi} \cite{gurobi}. Of course, the treatment of these problems \DSr{follows branch-and-bound techniques that might not finish within reasonable time, but they have very good performance in practice.}
	\end{remark}
	
	\DSr{\begin{remark}
			The above results (Propositions \ref{thm:WS} and \ref{thm:SWS}) can be easily adapted for the case where some initial distributions $x_{0,i}$ are zero. In these cases, as we mentioned, we would need to remove the exiting flow variables $\{v_{ij}\ :\ j\neq i\}$ from the \GL{follower}'s problem and \GL{analyse} the indexes $i\in I$ separately between the cases when $x_{0,i}$ is zero or not. We prefer to avoid this heavy notation \GL{in this model}. 
	\end{remark}}
	
	\section{Numerical Results and Discussion}\label{sec:Numerical}
	
	In order to compare the WS and SWS indicators and test the previously presented formulations, we designed some numerical experiments with artificial data. \DSr{These experiments are meant to be illustrative as a proof of concept. Due to the several simplifications \GL{applied to this model}, real-data experiments would not be valid, nor consistent. However, here we aim to see how EVSI behaves and \GL{how different WS and SWS can be}, even in this simplified setting. }\medskip

	We consider $I=\{1,2,3,4\}$ (four connected zones) who simulate four different communes of the Metropolitan Region of Chile (Santiago, Renca, Maipu, La Florida), $m=3$ (low, normal and high demand scenarios) and $c=0.75$ (75$\%$ of the ride fare is taken by the driver). Equation \eqref{eq:Def-Demand}, that represents the connection between the ride price on each zone and the respective effective demand, is considered with $\delta = 0.9$, and the ride prices are delimited on each zone with $p_{i,\min} = 2.5$ USD and $p_{i,\max} = 12.5$ USD. The costs $\alpha_{ij}$ are modeled considering the actual distance between the four communes and the price of gasoline in May 2021, i.e. $1.10$ USD per liter.
	
	\begin{table}[!ht]
		\begin{center}
			\begin{minipage}{0.5\textwidth}
				\centering 
				\begin{tabular}{|c|c|c|c|c|}
					\hline 
					Distances (km) & 1 & 2 & 3 & 4 \\
					\hline
					\hline
					1 & - & 11 & 17 & 20 \\
					\hline
					2 & 11 & - & 22 & 33  \\
					\hline
					3 & 17 & 22 & - & 18  \\
					\hline
					4 & 20 & 33 & 18 & -  \\
					\hline
				\end{tabular}
				\vspace{0.3cm}
				\caption{Distances considered between zones.}
			\end{minipage}
		\end{center}
	\end{table}
	
	A total fleet of $N_0 = 1000$ unmatched drivers is considered. For every $i\in I$, we generated 10 uniformly distributed values $h_{0,i}$ in the $[0,1]$ interval. From those values, we built samples for the nominal demand given by 
	\[
	\bar{d}_{0,i} = D_0h_{0,i},\quad \forall i\in I.
	\] 
	
	For each of these scenarios, we simulated a hundred samples for the triangular distributed nominal demand $d_{0,i}\sim \mathrm{Tri}(0.7\bar{d}_{0,i},1.3\bar{d}_{0,i})$. These simulations are combined with the following scenarios for the demand and the previously matched drivers:
	\begin{enumerate}
		\item We consider the aggregated demand coefficient $D_0$ as a function of the total fleet of unmatched drivers: that is, $D_0 = PN_0$, where $P=1,2,3,4,5$. \GL{We refer to $P$ as the demand multiplier}. This leads to the construction of the samples for $(d_{0,i}\, :\ i\in I)$.
		\item The quantity of previously matched drivers $y_j\sim U(0,\bar{y})$, with $\bar{y}$ as a proportion of the total fleet of unmatched drivers, $\bar{y} = QN_0$, with $Q = 0.25, 0.5, 0.75$. \GL{We refer to $Q$ as the fleet multiplier.}
	\end{enumerate}
	
	Therefore, we consider a total of 15 different scenarios as a combination of the values for $D_0$ and $\bar{y}$. Then, for each scenario, we consider 10 different nominal values for the vector $\bar{d}_0$ (given by the generated vectors $(h_{0,j})_{j\in I}$), and for each of those nominal values, we compute solutions for the WS and SWS reformulated problems for 100 different samples of the pairs $(y,d_0)$. Thus, for each scenario, we solved 1000 samples.

	\begin{figure}[ht]
		
		\begin{center}
			\begin{minipage}{0.7\textwidth}
				\hspace*{-1cm}   
				\centering
				\begin{tikzpicture}[scale=0.6]
					\pgfplotsset{
						xmin=1, xmax=5, legend pos = south east
					}
					
					\begin{axis}[
						axis y line*=left,
						/pgf/number format/1000 sep={},
						ymin=1000, ymax=5000,
						xlabel=Demand coefficient,
						ylabel={Total Income [USD]},
						xtick align = outside,
						x tick label style={
							/pgf/number format/.cd,
							fixed,
							fixed zerofill,
							precision=0,
							/tikz/.cd
						}
						]
						\addplot[color=blue,mark=o,dashed, thick, mark options={solid}] table [x index=0, y index=1, col sep=comma] {Y1.txt};
						\label{plot_Y1_WS}
						\addlegendentry{\small Wait-and-See}
						\addplot[color=red,mark=triangle,dashed, thick, mark options={solid}] table [x index=0, y index=2, col sep=comma] {Y1.txt};
						\label{plot_Y1_SWS}
						\addlegendentry{\small Shared-Wait-and-See}
					\end{axis}
				\end{tikzpicture}
				\caption{WS and SWS Comparison: $N_0 = 1000$, $\bar{y}=0.25N_0$}\label{fig:WS-SWS-1}
			\end{minipage}
		\end{center}
		\begin{minipage}{0.4\textwidth}
			\hspace*{-1cm}
			\begin{tikzpicture}[scale=0.6]
				\pgfplotsset{
					xmin=1, xmax=5, legend pos = south east
				}
				
				\begin{axis}[
					axis y line*=left,
					/pgf/number format/1000 sep={},
					ymin=1000, ymax=5000,
					xlabel=Demand coefficient,
					ylabel={Total Income [USD]},
					xtick align = outside,
					x tick label style={
						/pgf/number format/.cd,
						fixed,
						fixed zerofill,
						precision=0,
						/tikz/.cd
					}
					]
					\addplot[color=blue,mark=o,dashed, thick, mark options={solid}] table [x index=0, y index=1, col sep=comma] {Y2.txt};
					\label{plot_Y2_WS}
					\addlegendentry{\small Wait-and-See}
					\addplot[color=red,mark=triangle,dashed, thick, mark options={solid}] table [x index=0, y index=2, col sep=comma] {Y2.txt};
					\label{plot_Y2_SWS}
					\addlegendentry{\small Shared-Wait-and-See}
				\end{axis}
			\end{tikzpicture}
			\caption{WS and SWS Comparison: $N_0 = 1000$, $\bar{y}=0.5N_0$}
			\label{fig:WS-SWS-2}
		\end{minipage}
		\hfill
		\begin{minipage}{0.4\textwidth}
			\hspace*{-1cm}
			\begin{tikzpicture}[scale=0.6]
				\pgfplotsset{
					xmin=1, xmax=5, legend pos = south east
				}
				
				\begin{axis}[
					axis y line*=left,
					/pgf/number format/1000 sep={},
					ymin=1000, ymax=5000,
					xlabel=Demand coefficient,
					ylabel={Total Income [USD]},
					xtick align = outside,
					x tick label style={
						/pgf/number format/.cd,
						fixed,
						fixed zerofill,
						precision=0,
						/tikz/.cd
					}
					]
					\addplot[color=blue,mark=o,dashed, thick, mark options={solid}] table [x index=0, y index=1, col sep=comma] {Y3.txt};
					\label{plot_Y3_WS}
					\addlegendentry{\small Wait-and-See}
					\addplot[color=red,mark=triangle,dashed, thick, mark options={solid}] table [x index=0, y index=2, col sep=comma] {Y3.txt};
					\label{plot_Y3_SWS}
					\addlegendentry{\small Shared-Wait-and-See}
				\end{axis}
			\end{tikzpicture}
			\caption{WS and SWS Comparison: $N_0 = 1000$, $\bar{y}=0.75N_0$}
			\label{fig:WS-SWS-3}
		\end{minipage}
	\end{figure}   
	
	For each sample of $(y,d_0)$, we solved problems \eqref{eq:Problem-WS-MIP} and \eqref{eq:Reformulation-SWS-MIP} using \texttt{Gurobi v9.1.2} as solver \cite{gurobi}, and \texttt{Julia v1.6.2} as programming language \cite{bezanson2017julia}, with the extra constraint of having integer fluxes between the zones. \GL{All simulations were executed in a computer with an Intel Core(TM) i7-10700F processor, running at 2.90GHz, with 16 GB of RAM, running Windows 10 Pro}. \GL{The code and the exact samples we use are available at \url{https://github.com/dasalas22/EVSI-Ridehailing}.}\medskip
	
	Our results are displayed in Figures \ref{fig:WS-SWS-1}, \ref{fig:WS-SWS-2} and \ref{fig:WS-SWS-3}, which are organized as follow: each figure represents 5 scenarios, given by a \GL{fixed value of the fleet multiplier $Q$ and the five values of the demand multiplier $P$}. The plotted values are the mean values from the 1000 samples for the corresponding scenario: in blue dots, the values for the Wait-and-See problem \eqref{eq:Problem-WS-MIP}; in red triangles, the values for the Shared-Wait-and-See problem \eqref{eq:Reformulation-SWS-MIP}. \medskip
	
	As we can see in Figure \ref{fig:WS-SWS-1}, the SWS solution generates a greater total income for the scenarios with $\bar{y}=0.25N_0$, especially in the middle scenarios, when the aggregated demand coefficient does not reach extremely low or extremely high values. This behavior seems to be the same as the floating population of previously matched drivers increases, that is in Figures \ref{fig:WS-SWS-2} and \ref{fig:WS-SWS-3}, where $\bar{y}=0.5N_0$ and $\bar{y}=0.75N_0$ respectively. It seems that the $EVSI$, which is given by $SWS-WS$, is positive and has a concave behavior: when the aggregated demand is too small or too large, the $EVSI$ seems to be zero, reaching its maximum in a middle value. \medskip
	
	Another effect we observe in our results is that pricing alone is not enough to modify the behavior of the drivers. Their beliefs about the chances of getting a ride are also influential. Thus the revealing of such information has an important effect in the final reallocation.\medskip
	
	\GL{In what follows, we display the indicators of performance given by \texttt{Gurobi} on our experiments, that is, CPU time and optimality gap. We present the average results, aggregated by demand multiplier and fleet multiplier.}\medskip
	
	\GL{The average CPU times are displayed in Figures \ref{fig:WS-SWS-t1}, \ref{fig:WS-SWS-t2} and \ref{fig:WS-SWS-t3}, organized as the numerical results in the first figures, and summarized on Tables \ref{tab:WS-t} and \ref{tab:SWS-t}.}  \GL{Tables \ref{tab:WS-mog} and \ref{tab:SWS-mog} summarize the mean relative optimality gap reported by the optimization solver, considering all of the samples used for the numerical experiments, for the WS and SWS computation.}
	
	\begin{figure}[ht]
		
		\begin{center}
			\begin{minipage}{0.7\textwidth}
				\hspace*{-1cm}   
				\centering
				\begin{tikzpicture}[scale=0.6]
					\pgfplotsset{
						xmin=1, xmax=5, legend pos = north west
					}
					
					\begin{axis}[
						axis y line*=left,
						/pgf/number format/1000 sep={},
						ymin=0, ymax=750000,
						xlabel=Demand coefficient,
						ylabel={CPU Time [s]},
						xtick align = outside,
						x tick label style={
							/pgf/number format/.cd,
							fixed,
							fixed zerofill,
							precision=0,
							/tikz/.cd
						}
						]
						\addplot[color=blue,mark=o,dashed, thick, mark options={solid}] table [x index=0, y index=1, col sep=comma] {Y1_t.txt};
						\label{plot_Y1t_WS}
						\addlegendentry{\small Wait-and-See}
						\addplot[color=red,mark=triangle,dashed, thick, mark options={solid}] table [x index=0, y index=2, col sep=comma] {Y1_t.txt};
						\label{plot_Y1t_SWS}
						\addlegendentry{\small Shared-Wait-and-See}
					\end{axis}
				\end{tikzpicture}
				\caption{\GL{WS and SWS - CPU Time Comparison, $\bar{y}=0.25N_0$}}\label{fig:WS-SWS-t1}
			\end{minipage}
		\end{center}
		\begin{minipage}{0.4\textwidth}
			\hspace*{-1cm}
			\begin{tikzpicture}[scale=0.6]
				\pgfplotsset{
					xmin=1, xmax=5, legend pos = north west
				}
				
				\begin{axis}[
					axis y line*=left,
					/pgf/number format/1000 sep={},
					ymin=0, ymax=120000,
					xlabel=Demand coefficient,
					ylabel={CPU Time [s]},
					xtick align = outside,
					x tick label style={
						/pgf/number format/.cd,
						fixed,
						fixed zerofill,
						precision=0,
						/tikz/.cd
					}
					]
					\addplot[color=blue,mark=o,dashed, thick, mark options={solid}] table [x index=0, y index=1, col sep=comma] {Y2_t.txt};
					\label{plot_Y2t_WS}
					\addlegendentry{\small Wait-and-See}
					\addplot[color=red,mark=triangle,dashed, thick, mark options={solid}] table [x index=0, y index=2, col sep=comma] {Y2_t.txt};
					\label{plot_Y2t_SWS}
					\addlegendentry{\small Shared-Wait-and-See}
				\end{axis}
			\end{tikzpicture}
			\caption{\GL{WS and SWS - CPU Time Comparison, $\bar{y}=0.5N_0$}}
			\label{fig:WS-SWS-t2}
		\end{minipage}
		\hfill
		\begin{minipage}{0.4\textwidth}
			\hspace*{-1cm}
			\begin{tikzpicture}[scale=0.6]
				\pgfplotsset{
					xmin=1, xmax=5, legend pos = north west
				}
				
				\begin{axis}[
					axis y line*=left,
					/pgf/number format/1000 sep={},
					ymin=0, ymax=100000,
					xlabel=Demand coefficient,
					ylabel={CPU Time [s]},
					xtick align = outside,
					x tick label style={
						/pgf/number format/.cd,
						fixed,
						fixed zerofill,
						precision=0,
						/tikz/.cd
					}
					]
					\addplot[color=blue,mark=o,dashed, thick, mark options={solid}] table [x index=0, y index=1, col sep=comma] {Y3_t.txt};
					\label{plot_Y3t_WS}
					\addlegendentry{\small Wait-and-See}
					\addplot[color=red,mark=triangle,dashed, thick, mark options={solid}] table [x index=0, y index=2, col sep=comma] {Y3_t.txt};
					\label{plot_Y3t_SWS}
					\addlegendentry{\small Shared-Wait-and-See}
				\end{axis}
			\end{tikzpicture}
			\caption{\GL{WS and SWS - CPU Time Comparison, $\bar{y}=0.75N_0$}}
			\label{fig:WS-SWS-t3}
		\end{minipage}
	\end{figure}

	\begin{table}[!ht]
		\centering 
		\begin{tabular}{|c|c|c|c|}
			\hline 
			Demand Multiplier & $\bar{y}=0.25N_0$ & $\bar{y}=0.5N_0$ & $\bar{y}=0.75N_0$  \\
			\hline
			\hline
			1 & 366.95 & 229.59  & 187.97 \\
			\hline
			2 & 2284.78 & 524.38 & 306.82 \\
			\hline
			3 & 57459.26 & 6810.38 & 5253.85 \\
			\hline
			4 & 152732.21 & 84520.70 & 24285.40 \\
			\hline
			5 & 748790.72 & 118391.25 & 95443.21 \\
			\hline                
		\end{tabular}
		\vspace{0.3cm}
		\caption{\centering\GL{CPU Time for WS Computation [s].}}
		\label{tab:WS-t}
	\end{table}
	\begin{table}[!ht]
		\centering 
		\begin{tabular}{|c|c|c|c|}
			\hline 
			Demand Multiplier & $\bar{y}=0.25N_0$ & $\bar{y}=0.5N_0$ & $\bar{y}=0.75N_0$  \\
			\hline
			\hline
			1 & 116.91 & 43.61 & 18.79 \\
			\hline
			2 & 205.97 & 135.75 & 74.56 \\
			\hline
			3 & 560.21 & 167.47 & 110.70 \\
			\hline
			4 & 8679.34 & 756.37 & 188.07 \\
			\hline
			5 & 27514.56 & 2638.26 & 557.72 \\
			\hline                
		\end{tabular}
		\vspace{0.3cm}
		\caption{\centering\GL{CPU Time for SWS Computation [s].}}
		\label{tab:SWS-t}
	\end{table}
	\begin{table}[!ht]
		\centering 
		\begin{tabular}{|c|c|c|c|}
			\hline 
			Demand Multiplier & $\bar{y}=0.25N_0$ & $\bar{y}=0.5N_0$ & $\bar{y}=0.75N_0$  \\
			\hline
			\hline
			1 & 1.53E-04 & 1.13E-05 & 1.41E-05 \\
			\hline
			2 & 4.31E-04 & 1.12E-03	& 1.46E-03 \\
			\hline
			3 & 2.64E-02 & 6.15E-03	& 4.42E-03 \\
			\hline
			4 & 1.18E-01 & 4.14E-02 & 1.97E-02 \\
			\hline
			5 & 1.69E-01 & 8.97E-02 & 3.89E-02 \\
			\hline                
		\end{tabular}
		\vspace{0.3cm}
		\caption{\centering\GL{Mean Relative Gap for WS Computation.}}
		\label{tab:WS-mog}
	\end{table}
	\begin{table}[!ht]
		\centering
		\begin{tabular}{|c|c|c|c|}
			\hline 
			Demand Multiplier & $\bar{y}=0.25N_0$ & $\bar{y}=0.5N_0$ & $\bar{y}=0.75N_0$  \\
			\hline
			\hline
			1 & 8.09E-06 & 4.24E-06 & 5.73E-06 \\
			\hline
			2 & 8.57E-06 & 1.45E-05 & 6.95E-06 \\
			\hline
			3 & 7.98E-03 & 1.13E-05 & 1.16E-05 \\
			\hline
			4 & 1.36E-01 & 8.20E-03 & 9.31E-06 \\
			\hline
			5 & 2.58E-01 & 1.07E-01 & 2.08E-02 \\
			\hline                
		\end{tabular}
		\vspace{0.3cm}
		\caption{\centering\GL{Mean Relative Optimality Gap for SWS Computation.}}
		\label{tab:SWS-mog}
	\end{table}
	\newpage
	\GL{By looking at the results above, we deduce that the WS problems are a lot harder than the SWS problems, at least for the formulations presented in this work. This is consistent with Remark \ref{rem:bilinearity}, where we observed that the WS final reformulation carried polynomial equality constraints, which are absent in the couterpart of SWS. Also, even though optimality gaps are always less that 0.3\%, it is worth to note that we obtain less accurate results for the demand multipliers $P=3,4,5$. This is consistent with the results of CPU times. Probably, it means that the allocation problems become harder for those multipliers. }\medskip
	
	Finally, as an illustrative example, we show how different the WS and the SWS solutions are for a particular scenario, given by one of the numerical experiments we did. In this scenario, $\bar{y} = 750$ and $P = 5$, and the values for $x_0$, $y$ and $d_0$ are the ones described in Table \ref{tab:Values_Scenario}.
	
	\begin{table}[ht]
		\begin{center}
			\begin{minipage}{0.7\textwidth}
				\centering
				\hspace{-1cm}
				\begin{tabular}{c|c|c|c}
					Zone & $x_0$ & $y$ & $d_0$  \\
					\hline
					1   & 107& 444.55 & 510\\
					2 & 283 & 296.68 & 1945\\
					3 & 399& 420.20 & 1010\\
					4 & 211 & 568.07 & 1535\\
					\hline 
				\end{tabular}
				\vspace{0.3cm}
				\caption{Particular scenario with $\bar{y} = 750$ and $P = 5$.}
				\label{tab:Values_Scenario}
			\end{minipage}
		\end{center}
	\end{table}
	
	The WS and SWS solutions are presented in Figures \ref{fig:WS-flow} and \ref{fig:SWS-flow}, respectively. The missing edges in each graph have $0$-flow. The value functions for each solution are
	$\psi(y,d_0) = 3935.4 \mbox{ USD}$ and $\varphi(y,d_0) = 4635.4 \mbox{ USD}$.
	
	\begin{figure}[ht]
		\begin{minipage}{0.4\textwidth}
			\centering
			\begin{tikzpicture}[->,>=stealth',shorten >=0pt,auto,node distance=2cm,
				semithick,square/.style={regular polygon,regular polygon sides=4}]
				
				\tikzstyle{every state}=[text=black,font = \bfseries, fill opacity = 0.35, text opacity = 1]
				\node[state] 		 (Z1) [fill=gray]                   					{1};
				\node[state] 		 (Z2) [fill=gray,above right of=Z1]       					{2};
				\node[state] 		 (Z4) [fill=gray,below right of=Z1]       					{4};
				\node[state] 		 (Z3) [fill=gray,above right of=Z4]       					{3};
				
				\node (X1) [below of=Z1,  node distance = 0.85cm]                   					{\small $\begin{array}{l}x_1 = 107\\ p_1 = 6.85\end{array}$};
				\node (X2) [above of=Z2,  node distance = 0.85cm]                   					{\small $\begin{array}{l}x_2 = 283\\ p_2 = 9.07\end{array}$};
				\node (X3) [below of=Z3,  node distance = 0.85cm]                   					{\hspace{0.5cm}\small $\begin{array}{l}x_3 = 93\\ p_3 = 6.05\end{array}$};
				\node (X4) [below of=Z4,  node distance = 0.85cm]                   					{\small $\begin{array}{l}x_4 = 517\\ p_4 = 8.46\end{array}$};
				
				\path 
				(Z3) edge node[sloped]{$306$} (Z4)
				;
			\end{tikzpicture}
			\caption{Graph with perfect information (Wait-and-see solution). Optimal value: 3935.4 USD.}
			\label{fig:WS-flow}
		\end{minipage}
		\hfill
		\begin{minipage}{0.4\textwidth}
			
			\centering	
			\begin{tikzpicture}[->,>=stealth',shorten >=0pt,auto,node distance=2cm,
				semithick,square/.style={regular polygon,regular polygon sides=4}]
				
				\tikzstyle{every state}=[text=black,font = \bfseries, fill opacity = 0.35, text opacity = 1]
				\node[state] 		 (Z1) [fill=gray]                   					{1};
				\node[state] 		 (Z2) [fill=gray,above right of=Z1]       					{2};
				\node[state] 		 (Z4) [fill=gray,below right of=Z1]       					{4};
				\node[state] 		 (Z3) [fill=gray,above right of=Z4]       					{3};
				
				\node (X1) [below of=Z1,  node distance = 0.85cm]                   					{\small $\begin{array}{l}x_1 = 0\\ p_1 = 6.80\end{array}$};
				\node (X2) [above of=Z2,  node distance = 0.85cm]                   					{\small $\begin{array}{l}x_2 = 613\\ p_2 = 9.50\end{array}$};
				\node (X3) [below of=Z3,  node distance = 0.85cm]                   					{\hspace{0.5cm}\small $\begin{array}{l}x_3 = 1\\ p_3 = 6.94\end{array}$};
				\node (X4) [below of=Z4,  node distance = 0.85cm]                   					{\small $\begin{array}{l}x_4 = 386\\ p_4 = 9.54\end{array}$};
				
				\path (Z1) edge node[sloped]{$107$} (Z4)
				(Z3) edge node[sloped]{$330$} (Z2)
				(Z3) edge node[sloped]{$68$} (Z4)
				;
			\end{tikzpicture}
			\caption{Graph with shared information (Shared wait-and-see solution). Optimal value: 4635.4 USD.}
			\label{fig:SWS-flow}
		\end{minipage}
	\end{figure} 
	
	The reader can appreciate that the revelation of information produces a huge change in the solution. For this scenario, Zone 2 has a huge demand, but drivers simply do not know it since this value is vastly different of what they usually observe. With the second reallocation (Figure \ref{fig:SWS-flow}), the ride-hailing company can increase the prices in zones 2 and 4.
	
	
	\section{Final comments}\label{sec:Final}
	
	In this work, we presented a new indicator, the \emph{Expected Value of Shared Information}, that allows to measure the value of sharing information in the context of Stackelberg games. This indicator is relevant in problems where the leader has more information with respect to the follower. \medskip
	
	\DSr{We used this indicator to study the value of sharing information in the context of ride-hailing companies, considering the demand and part of the behavior of matched drivers as the asymmetric information. We then studied the problem of allocation: the ride-hailing company (leader) decides the spatial prices, while the unmatched drivers (the followers) decide their allocations. \medskip
		
		The main contribution of this work is the definition of the EVSI and the proof-of-concept in a simplified version of the allocation problem of unmatched drivers. We provided sound mathematical developments allowing us to get tractable formulations for numerical experiments. While our numerical results, coming from simulations with randomly generated data, strongly suggest that sharing information might be beneficial for the leader, we can only conclude that a more detailed study should be conducted and that the EVSI deserves attention in the context of ride-hailing.\medskip
		
		As a first work dealing with this new indicator in the context of ride-hailing companies, several simplifying assumptions where made: 1) we studied only the one-stage problem; 2) we simplified the drivers' equilibrium problem into a single welfare optimization problem; 3) we worked with artificially generated data; and 4) we assumed that the leader had access to a perfect forecast of the demand. However, motivated by the promising results we obtained here, we aim to study a more complex model considering: non-cooperative drivers; multiple stages; partial information sharing; data-driven distributions (based either on real-data or benchmark data available in the literature). Over such a development, the numerical experiments should be conducted over benchmark datasets, such as \cite{networkData}. All these challenges were out of the scope of this work, since our main objective, that we hope we have achieved, was to validate the EVSI as a tractable and pertinent indicator to measure the value of sharing information in ride-hailing. }

	
	\paragraph{Acknowledgements:} This work was partially supported by Center of Mathematical Modelling, FB210005, BASAL funds for centers of excellence from ANID-Chile, and by MathAmsud program through the project MATHAMSUD 20-MATH-08. \GL{The second author was partially funded by ANID-Chile through the grant FONDECYT 11220586. The authors thank the anonymous referees for their valuable insight during the revision process, that help us to substantially improved the first version of this work.}
	

	\vspace{0.5cm}
	
	\noindent Gianfranco Liberona
	
	\medskip
	
	\noindent Universidad de
	O'Higgins\newline Av. Libertador Bernardo O'Higgins 611, Rancagua, Chile
	\smallskip
	
	\noindent E-mail: \texttt{gianfranco.liberona@uoh.cl} \newline\vspace{0.4cm}
	
	\noindent David Salas
	
	\medskip
	
	\noindent Instituto de Ciencias de la Ingenieria, Universidad de
	O'Higgins\newline Av. Libertador Bernardo O'Higgins 611, Rancagua, Chile
	\smallskip
	
	\noindent E-mail: \texttt{david.salas@uoh.cl} \newline\noindent
	\texttt{http://davidsalasvidela.cl} \medskip
	
	\noindent Research supported by the grants: \smallskip\newline%
	\textsc{FONDECYT 11220586} (\textsc{ANID}-Chile)\newline
	\textsc{MathAmsud 20-MATH-08}
	\newline\textsc{CMM
		FB210005 BASAL} funds for centers of excellence (\textsc{ANID}-Chile)
	\newline\vspace{0.4cm}
	
	\noindent Léonard von Niederh\"{a}usern
	
	\medskip
	
	\noindent Centro de Modelamiento Matemático CNRS IRL 2807, Universidad de Chile,\newline
	Beauchef 851, Santiago, Chile. (on leave)\newline
	
	\noindent E-mail: \texttt{leonard.vonniederhausern@uoh.cl}\medskip
	
	\noindent Research supported by the grants: \smallskip\newline
	\textsc{MathAmsud 20-MATH-08}\newline
	\textsc{CMM
		FB210005 BASAL} funds for centers of excellence (\textsc{ANID}-Chile)
	
\end{document}